\documentclass[12pt]{article}

\usepackage{iftex}
\ifpdftex
    \usepackage{CJKutf8}
    \AtBeginDocument{\begin{CJK*}{UTF8}{gbsn}}
    \AtEndDocument{\end{CJK*}}
\else
    \usepackage[scheme=plain]{ctex}
    
    \newenvironment{CJK*}[2]{}{}
\fi

\usepackage{amsmath,amsthm,amssymb,amsfonts,mathtools,thmtools}
\usepackage[numbers,sort&compress]{natbib}
\usepackage{color,colordvi}
\usepackage{soul}
\usepackage{fullpage}

\usepackage[letterpaper,top=2cm,bottom=2cm,left=3cm,right=3cm,marginparwidth=2cm]{geometry}

\usepackage[colorlinks=true, allcolors=blue]{hyperref}
\usepackage[capitalise]{cleveref}
\usepackage{standalone}
\usepackage{indentfirst}
\usepackage[affil-it]{authblk}
\usepackage{diagbox}
\usepackage{changes}
\usepackage{enumitem}
\usepackage{bbm}
\usepackage{orcidlink}

\newtheorem{theorem}{Theorem}[section]
\newtheorem{lemma}[theorem]{Lemma}

\newtheorem{proposition}[theorem]{Proposition}
\newtheorem{claim}[theorem]{Claim}

\newtheorem{problem}[theorem]{Problem}

\newtheorem{observation}[theorem]{Observation}

\newtheorem*{claim*}{Claim}

\newcommand{\textotherwise}{\text{otherwise}}

\newcommand{\QQ}{\mathbb{Q}}
\newcommand{\PP}{\mathbb{P}}

\newcommand{\ZZ}{\mathbb{Z}}

\newcommand{\cG}{\mathcal{G}}
\newcommand{\cQ}{\mathcal{Q}}

\DeclareMathOperator{\diag}{diag}

\DeclareMathOperator{\CAN}{CAN}
\DeclareMathOperator{\sCAN}{semi-CAN}
\DeclareMathOperator{\FRI}{FRI}
\DeclareMathOperator{\FCI}{FCI}
\DeclareMathOperator{\IRI}{IRI}
\DeclareMathOperator{\ICI}{ICI}

\DeclarePairedDelimiter{\abs}{\lvert}{\rvert}
\DeclarePairedDelimiter{\card}{\lvert}{\rvert}
\DeclarePairedDelimiter{\set}{\lbrace}{\rbrace}

\DeclarePairedDelimiter{\paren}{\lparen}{\rparen}

\DeclarePairedDelimiter{\floor}{\lfloor}{\rfloor}
\DeclarePairedDelimiter{\ceil}{\lceil}{\rceil}

\title{Almost all graphs have no cospectral mates with height relative small to its order}

\author{Da Zhao\footnote{email: zhaoda@ecust.edu.cn}~\orcidlink{0000-0002-9582-0778}}
\affil{School of Mathematics, East China University of Science and Technology, 130 Meilong Road, Shanghai 200237, China.}

\date{}

\begin{document}

\maketitle

\begin{abstract} 
    The main result of this paper shows that almost all graphs of order $n$ have no cospectral mates with height $o(( n / \ln n)^{1/10})$, improving an earlier result on cospectral mates with fixed level and covering the cospectral graphs obtained by GM-switching of relative small blocks. 
\end{abstract}




\section{Introduction}

In 1966, Kac~\cite{kac1966CanOneHeara} asked whether the Laplacian eigenvalues characterize a plane region uniquely in his widespread paper 'Can one hear the shape of a drum?', and he attributed this question to Bochner. 
Gordon--Webb--Wolpert~\cite{gordon1992OneCannotHear} answered the question in the negative in 1992, namely one can not hear the shape of a drum. 
Note that Milnor~\cite{milnor1964EigenvaluesLaplaceOperator} constructed two non-isometric 16-dimensional flat tori that are isospectral in 1964.

A similar question was considered in the discrete setting from a different origin. 
In 1956, during the study of Hückel molecular orbital (HMO) theory, chemists Günthard and Primas~\cite{gunthard1956ZusammenhangGraphentheorieUnd} asked whether the adjacency spectrum of a graph characterizes it up to isomorphism. 
Collatz and Sinogowitz~\cite{voncollatz1957SpektrenEndlicherGrafen} provided counterexamples in the next year, though it is not clear to the author whether their motivation is to answer the question by Günthard and Primas. 

However, the occurrences of the above phenomenon might be rare. 
In 1979, another paper by Wolpert~\cite{wolpert1979LengthSpectraModuli} showed that the eigenvalues of the Laplacian operator do characterize a domain in the generic case, where it reads `Consider a compact Riemann surface $S$ of genus $g$, $g > 2$ endowed with the Poincare metric. I.M. Gel'fand posed the following question. Do the
eigenvalues of the Laplace Beltrami operator determine the geometry of $S$?'. 
We do not find a reference indicating when and where Gel'fand posed such a question. 

In 2003, van Dam and Haemers~\cite{vandam2003WhichGraphsAre} wrote 'whether it is possible to define the matrix of $G$ in a (not so sensible) way such that every graph becomes DS.', where 'DS' means 'determined by spectrum'. 
They also wrote that 'it is more likely that almost all graphs are DS, than that almost all graphs are non-DS'. 
Their subsequent paper~\cite{vandam2009DevelopmentsSpectralCharacterizations} says that the work by Wang and Xu~\cite{wang2006SufficientConditionFamilya,wang2006ExcludingAlgorithmTesting,wang2007NoteGeneralizedSpectral} `strengthens our belief that the statement ``almost all graphs are not DS'' (which is true for trees) is false.'.
They did not use the word conjecture in the paper. 
Haemers used it in a number of conference talks. 
The first paper mentioning Haemers conjecture might be~\cite{brouwer2009CospectralGraphs12} by Brouwer and Spence.

A similar question was considered in random matrix theory as well. 
Vu~\cite{jang2014InvitedLectures,vu2021RecentProgressCombinatorial} raised the fingerprint conjecture in ICM 2014, namely almost all symmetric $\pm 1$ matrices are determined by their spectrum up to switching. 

There have been many studies on cospectral graphs and graphs determined by spectrum (DS graphs). 
The reader may refer to the surveys~\cite{vandam2003WhichGraphsAre,vandam2009DevelopmentsSpectralCharacterizations}.
The phenomenon on trees is the opposite, namely almost all trees have cospectral mates~\cite{harary1973NewDirectionsTheory}, proved by Schwenk. 
Recently Wang and Huang~\cite{wang2025AlmostAllCographs} proved that almost all cographs have a cospectral mate.
Van de Berg and Van Werde~\cite{vandeberg2026AreSparseGraphs} proved that the largest component in a random sparse graph admits a cospectral mate almost surely.
The phenomenon that the edge density of a graph affects the probability of existence of cospectral mates was observed in~\cite{brouwer2009CospectralGraphs12}.
Current best lower bound on the number of DS graphs is given by Koval and Kwan~\cite{koval2024ExponentiallyManyGraphs}, they proved that there are at least $\exp(cn)$ graphs with $n$ vertices are DS for some $c > 0$, improving an earlier bound of order $\exp(c \sqrt{n})$. 
On the other hand, Haemers and Spence~\cite{haemers2004EnumerationCospectralGraphs} proved that there are at least $(\frac{1}{24} - o(1)) n^3 g_{n-1}$ non-DS graphs with $n$ vertices, where $g_n$ is the number of non-isomorphic graphs with $n$ vertices. 
However, the fraction of known non-DS graphs and the fraction of known DS graphs both tend to $0$ as the order of the graph goes to infinity. 

Given a graph $G$, we denote by $A_G$ its adjacency matrix. 
Two graphs $G$ and $H$ are called cospectral mates if $G$ and $H$ are not isomorphic and the eigenvalues of $A_G$ and $A_H$ are identical.
In other words, there exists an orthogonal matrix $Q$ such that $Q^\top A_G Q = A_H$. 
Two graphs $G$ and $H$ are called cospectral mates of level $\ell \in \ZZ$ if $G$ and $H$ are not isomorphic and there exists a rational orthogonal matrix $Q$ such that $Q^\top A_G Q = A_H$ and $\ell Q$ is integral. 
Two graphs $G$ and $H$ are called generalized cospectral mates if $G$ and $H$ are not isomorphic and the eigenvalues of $A_G$ and $A_H$ are identical as well as the eigenvalues of the complement graphs $\overline{G}$ and $\overline{H}$.
Such generalized spectrum was considered by Tutte~\cite{MR538033}, which was later named as idiosyncratic polynomial of graphs. 

Recently Wei Wang and the author~\cite{wang2025AlmostAllGraphs} showed that almost all graphs in $\cG(n, p)$ with $0<p<1$ have no cospectral mates of level $\ell$ for every fixed level $\ell \geq 2$, generalizing an earlier result~\cite{wang2010AsymptoticBehaviorGraphs} stating that almost all graphs in $\cG(n, 1/2)$ have no generalized cospectral mates of level $2$. 
It is known that almost all graphs are controllable~\cite{orourke2016ConjectureGodsilConcerning}.
A sufficient condition that almost all graphs are determined by their generalized spectrum is
\begin{align}\label{eq:across}
    &\Pr(G \text{ is not controllable}) \nonumber\\
    + &\sum_{\ell = 2}^{n^{n^2}} \Pr(G \text{ has a generalized cospectral mate with level } \ell) \to 0
\end{align}
as $n \to \infty$. 
Their proof certifies that every summand on the left-hand side tends to $0$. 
Note that the upper bound $n^{n^2}$ of the level in~\cref{eq:across} is far from optimal. 
The problem that what is the true upper bound of the level is interesting in its own right.
Van Werde~\cite{werde2026ExactCospectralityProbabilities} generalized their result and computed the exact cospectral probability for uniform random matrices over an algebraic field. 

Another approach focuses on the generalization of spectrum to spectrum of other matrices. 
Wei Wang and the author~\cite{wang2026GraphIsomorphismMultivariate} introduced the multivariate graph spectrum, which covers the adjacency spectrum, the Laplacian spectrum, the generalized spectrum, etc.
They related the problem to the graph isomorphism problem, where tools such as Weisfeiler--Lehman algorithm and coherent configuration (association scheme) are well studied. 
Naturally they considered the generalized block Laplacian spectrum and exhibited evidence that such spectrum characterizes a graph almost surely. 
Later Xu and the author~\cite{xu2025GeneralizedBlockDiagonal} reduced the generalized block Laplacian spectrum to generalized block diagonal Laplacian spectrum, where the matrix is symmetric and the eigenvalues are real. 
Wang and Wu~\cite{wang2025GeneralizationJohnsonNewmanTheorem} further generalized to perturbations with multiple rank-one matrices. 
Independently, Xiang~\cite{xiang2026NaturalGraphSpectra} showed that there exists a natural graph spectrum determining almost all graphs, which answers the question raised by van Dam and Haemers by showing the existence of a sensible matrix. 

The main result of this paper is an improvement of~\cite{wang2025AlmostAllGraphs} via replacing level by height. 
Given a rational matrix $M$, the height of $M$ is the largest positive integer among the denominators of all entries in $M$. 
Clearly a rational matrix of level $\ell$ is of height at most $\ell$. 
The motivation of introducing height comes from the following fact. 
A lot of known cospectral graphs are obtained by Godsil--McKay switching and its generalizations, Wang--Qiu--Hu switching, Abiad--Haemers switching, and design switching~\cite{godsil1982ConstructingCospectralGraphs,wang2019CospectralGraphsGMswitching,abiad2024SwitchingMethodsLevel,ihringer2025DesignSwitchingGraphsa}. 
Consider a GM-switching with block sizes $2p_1, 2p_2, \ldots, 2p_t$, where these $p_i$'s are distinct prime numbers. 
Then the level of the switching matrix grows pretty fast while its height is still bounded by $n$. 
The use of height also relieves us from the summation in~\cref{eq:across}.

\begin{theorem}\label{thm:fixed-height}
    Almost all graphs have no cospectral mates of fixed height $h \geq 2$. 
    Namely, let $n$ be a positive integer, $0 < p < 1$, and $G \sim \cG(n, p)$ a random Erdős--Rényi graph, then
    \begin{align}
        \Pr(G \text{ has a cospectral mate with height} \leq h) \to 0
    \end{align}
    as $n \to \infty$ for every fixed $h \geq 2$.
\end{theorem}

In fact, the height $h$ could grow with $n$. 

\begin{theorem}\label{thm:growing-height}
    Almost all graphs of order $n$ have no cospectral mates of height $h = o\paren*{(n/\ln n)^{1/10}}$. 
    Namely, let $n$ be a positive integer, $0 < p < 1$, $G \sim \cG(n, p)$ a random graph, and $h = o\paren*{(n/\ln n)^{1/10}}$, then
    \begin{align}
        \lim_{n \to \infty} \Pr(G \text{ has a cospectral mate with height} \leq h) = 0. 
    \end{align}
\end{theorem}

We give the proofs in the subsequent two sections. 
Then we propose problems for future research in the discussion section.

\section{Proof of~\cref{thm:fixed-height}}

A \emph{graph} is a tuple $G = (V, E)$, where $V$ is a finite \emph{vertex} set and $E \subseteq \binom{V}{2}$ is the \emph{edge} set. 
An edge $\set{u, v} \in E$ is often denoted as $uv$ or $u \sim v$ if the graph is clear from context. 
The \emph{adjacency matrix} of $G$ is defined by $A = A_G = (a_{u, v})_{u, v \in V}$, where
\begin{align}
    a_{u,v} = 
    \begin{cases}
        1, & u \sim v \text{ is an edge}, \\
        0, & \textotherwise.
    \end{cases}
\end{align}
The \emph{spectrum} of a graph $G$ is the multiset of all eigenvalues of $A_G$. 

Let $S_n(\ZZ)$ be the set of symmetric matrices of order $n$ over $\ZZ$;
let $M_n(\ZZ)$ be the set of square matrices of order $n$ over $\ZZ$; 
let $O_n(\QQ)$ be the set of orthogonal matrix of order $n$ over $\QQ$. 
Given a rational matrix $M$, the \emph{height} $h(M)$ of $M$ is the largest denominator $h$ among all entries of $M$. 
For every positive integer $h$, let $O_{n, h}(\QQ)$ be the set of rational orthogonal matrices with height $h$. 
Note that the rational orthogonal matrices with height $1$ are exactly the signed permutation matrices. 

Two graphs $G_1 = (V_1, E_1)$ and $G_2 = (V_2, E_2)$ are called \emph{isomorphic} if there exists a bijection $\pi : V_1 \to V_2$ such that $uv \in E_1$ if and only if $\pi(u)\pi(v) \in E_2$. 
Two non-isomorphic graphs are called \emph{cospectral} if their spectra are identical; and we say that one graph is a \emph{cospectral mate} of the other. 
In particular, this implies that there exists an orthogonal matrix $Q$ such that $Q^\top A Q = B$, where $A$ and $B$ are the adjacency matrices of the two graphs. 
In general the orthogonal matrix $Q$ may not be rational. 
In the case that $Q$ is a rational orthogonal matrix with height $h$, we say that one graph is a cospcetral mate of the other with height $h$. 

We consider a random graph model $\cG(n, p)$. 
Namely, $G \sim \cG(n, p)$ is a graph on $n$ vertices with each edge chosen independently with probability $p \in (0, 1)$. 
For convenience, we denote by $\hat{p} = \max\set{p, 1-p}$ in the rest of the paper. 

The proof of~\cref{thm:fixed-height} is divided into four parts. 
Firstly, we reduce the rational orthogonal matrix to a canonical form. 
Secondly, for a fixed $Q$ of canonical form, we estimate the probability that $\Pr(Q^\top A Q \in M_n(\ZZ))$. 
Thirdly, we estimate the number of $Q$ in a canonical form. 
Lastly, we combine the estimates to prove the main theorem. 

For $Q \in O_{n, h}(\QQ)$, there exists a pair of signed permutations matrices $(P_R, P_C)$ such that $P_R^\top Q P_C$ is of the block diagonal form
\begin{align}
    \begin{bmatrix}
        Q_s & \\
         & I_{n-s}
    \end{bmatrix},
\end{align}
where $Q_s$ is a rational orthogonal matrix of order $s$ and height $h$ with entries being zero or fractions.  
We call it a \emph{canonical form} of $Q$. 
Note that the canonical form a rational orthogonal matrix is not unique. 
We denote by $\CAN(n, h; s)$ the subset of $O_{n, h}(\QQ)$ of block diagonal form $\diag(Q_s, I_{n-s})$, and by $\sCAN(n, h; s)$ the subset of $O_{n, h}(\QQ)$ of block diagonal form $\diag(Q_s, P_{n-s})$, where $Q_s$ is a rational orthogonal matrix of order $s$ with entries being zero or fractions, and $P_{n-s}$ is a signed permutation matrix of order $n-s$. 

\begin{lemma}\label{lem:bottle}
    Fix $Q \in \CAN(n, h; s)$. 
    Let $A$ be the adjacency matrix of a random graph $G \sim \cG(n, p)$. 
    Then
    \begin{align}
        \Pr(Q^\top A Q \in M_n(\ZZ)) \leq \hat{p}^{\frac{s}{2 h^4} \cdot (\frac{s}{2h^4} + n - s - 1)}.
    \end{align}
\end{lemma}

\begin{proof}
    
    Fix $Q \in \CAN(n, h; s)$. 
    For a column index $j$ of $Q$, we define $K(j)$ to be the set of row indices of non-zero entries of the $j$-th column of $Q$. 
    Next we shall choose two subsets of indices $I \subseteq \set{1,2, \ldots, s}$ and $J \subseteq \set{1,2, \ldots, n}$ such that the following hold.
    \begin{enumerate}
        \item It holds $K(i) \cap K(j) = \emptyset$ for $i, j \in I \cup J, i \neq j$. 
        \item The Cartesian products $K(i) \times K(j), i \in I, j \in J$ are disjoint subsets of $\set{1,2, \ldots, s} \times \set{1,2, \ldots, n}$;
        \item The sizes of $I$ and $J$ are large enough in the sense that $\card{I} = \Omega(s)$ and $\card{J} = \Omega(n)$.
    \end{enumerate}
    
    For every index $i \in \set{1,2, \ldots, s}$, we define the set $N(i) = \set{j \in \set{1,2, \ldots, s} : K(j) \cap K(i) \neq \emptyset}$. 
    
    \begin{claim}
        For every $i \in \set{1,2, \ldots, s}$, we have $\card{N(i)} \leq h^4$. 
    \end{claim}
    
    \begin{proof}
        Note that there are at most $h^2$ non-zero entries in each row or column of $Q$. 
        Therefore, $\card{K(i)} \leq h^2$. 
        For every $k \in K(i)$, there exist at most $h^2$ non-zero entries in the $k$-th row of $Q$. 
        Hence, $L(k) \coloneqq \card{\set{j \in \set{1,2, \ldots, s} : k \in K(j)}} \leq h^2$.
        Since $N(i) = \cup_{k \in K(i)} L(k)$, we get $\card{N(i)} \leq h^2 \times h^2 = h^4$.
    \end{proof}
    
    Next we define $I$ and $J$ iteratively. 
    Define $S_1 = \set{1,2, \ldots, s}$. 
    We take an arbitrary element $i_1 \in S_1$, and put the index $i_1$ into $I$. 
    We define $S_2 = S_1 \setminus N(i_1)$. 
    If $S_2 \neq \emptyset$, then we take an arbitrary element $i_2 \in S_2$, and put the index $i_2$ into $J$. 
    We define $S_3 = S_2 \setminus N(i_2)$. 
    If $S_3 \neq \emptyset$, then we take an arbitrary element $i_3 \in S_3$, and put the index $i_3$ into $I$. 
    We repeat the above process until $S_{t+1} = \emptyset$ for some positive integer $t$.
    Finally, we put $\set{s+1, s+2, \ldots, n}$ to $J$. 
    
    \begin{claim}
        We have $\card{I} \geq \ceil{\frac{1}{2}\ceil{\frac{s}{h^4}}} \geq \frac{s}{2h^4}$, and $\card{J} \geq \floor{\frac{1}{2}\ceil{\frac{s}{h^4}}} + n-s \geq \frac{s}{2h^4} + n - s - 1$.
    \end{claim}
    
    \begin{proof}
        Since $\card{N(i)} \leq h^4$ for every $i \in \set{1,2, \ldots, s}$, we have $S_t \neq \emptyset$ for $t = \ceil{\frac{s}{h^4}}$. 
        Hence, $\card{I} \geq \ceil{\frac{t}{2}}$, and $\card{J} \geq \floor{\frac{t}{2}} + n-s$. 
    \end{proof}
    
    Let $q_j, j \in \set{1,2, \ldots, n}$ be the $j$-th column of $Q$. 
    Let $q_{i,j}$ be the $i$-th row entry of $q_i$. 
    Let $b_{i,j} \coloneqq q_i^\top A q_j$ be the $(i,j)$-entry of $Q^\top A Q$. 
    
    \begin{proposition}
        It holds that
        \begin{align}
            \Pr(b_{i,j} \in \ZZ) \leq \hat{p}
        \end{align}
        for $i \in I, j \in J$. 
        And the $b_{i,j}$'s are independent random variables for $i \in I, j \in J$. 
    \end{proposition}
    
    \begin{proof}
        Note that
        \begin{align}
            b_{i, j} &= (Q^\top A Q)_{i, j} \\
                &= \sum_{1 \leq u, v \leq n} q_{u,i} a_{u, v} q_{v, j} \\
                &= \sum_{u \in K(i), v \in K(j)} a_{u, v} q_{u,i} q_{v, j}. 
        \end{align}
        So $b_{i, j}$ depends only on the variables $a_{u, v}$ for $(u, v) \in K(i) \times K(j)$. 
        Since $K(i) \times K(j), i \in I, j \in J$ are disjoint subsets of $\set{1,2, \ldots, s} \times \set{1,2, \ldots, n}$ and $K(i) \cap K(j) = \emptyset$ for $i, j \in I \cup J$ with $i \neq j$, we have that $b_{i,j}$ are independent for $i \in I, j \in J$. 
    
        Let $X = (x_{u, v})_{u \in K(i), v \in K(j)}$ be a $(0, 1)$-array with indices $K(i) \times K(j)$. 
        We define an involution on all such arrays.  
        Fix $u_0 \in K(i), v_0 \in K(j)$. 
        Define $X' = (x_{u, v}')_{u \in K(i), v \in K(j)}$, where
        \begin{align}
            x_{u, v}' = 
            \begin{cases}
                x_{u, v}, & (u, v) \neq (u_0, v_0), \\
                1-x_{u, v}, & (u, v) = (u_0, v_0).
            \end{cases}
        \end{align}
        Note that $0 < q_{u_0, i} < 1$ and $0 < q_{v_0, j} \leq 1$. 
        So if $b_{i, j} \in \ZZ$ for $A|_{K(i) \times K(j)} = X$, then $b_{i, j} \notin \ZZ$ for $A|_{K(i) \times K(j)} = X'$. 
        Namely,
        \begin{align}
            \Pr(b_{i,j} \in \ZZ \mid a_{u, v} = x_{u, v}, (u, v) \neq (u_0, v_0)) \leq \max(p, 1-p) = \hat{p}. 
        \end{align}
        By total probability theorem we get
        \begin{align}
            \Pr(b_{i,j} \in \ZZ) \leq \hat{p}. 
        \end{align}
    \end{proof}

    For large $n$, we have
    \begin{align}
        \Pr(Q^\top A Q \in M_n(\ZZ))
        &\leq \prod_{i \in I, j \in J} \Pr(b_{i, j} \in \ZZ) \\
        &\leq \hat{p}^{\card{I} \cdot \card{J}} \\
        &\leq \hat{p}^{\frac{s}{2 h^4} \cdot (\frac{s}{2 h^4} + n - s - 1)}.
    \end{align}
\end{proof}

Next we estimate the number of $Q$ in a canonical form. 

\begin{lemma}
    Let $h$ be a positive integer. 
    Then
    \begin{align}
        \card{\CAN(n, h; s)} \leq (2s h^2)^{s h^2}.
    \end{align}
\end{lemma}

\begin{proof}
    Let $Q$ be a canonical form in $\CAN(n, h; s)$. 
    For each non-zero entry in $Q$ other than $\pm 1$, it takes value in $\pm \frac{p}{q}$ with $0 < \abs{p} < q \leq h$. 
    There are $2 + \sum_{q=2}^h 2q \leq h^2 + h$ possibilities. 
    Since the sum of squares of entries in each column is $1$ and each non-zero entry contributes at least $1/h^2$, there are at most $h^2$ non-zero entries in each row. 
    The number of feasible column vector in the first $s$ columns is at most $\binom{s}{h^2} (h^2 + h)^{h^2} \leq (2sh^2)^{h^2}$. 
    Since the first $s$ columns determine the canonical form, we have $\card{\CAN(n, h; s)} \leq (2s h^2)^{s h^2}$. 
\end{proof}

For $A \in S_n(\ZZ)$, define 
\begin{align}
    \cQ(A) \coloneqq \set{Q \in O_n(\QQ) \mid Q^\top A Q \in M_n(\ZZ)}.
\end{align}

\begin{observation}
    Let $A \in S_n(\ZZ)$.
    Suppose $Q \in \cQ(A)$. 
    Then $QP \in \cQ(A)$ for every signed permutation matrix $P$ of order $n$. 
\end{observation}

For $Q \in O_n(\QQ)$, we denote by $\FRI(Q)$ ($\FCI(Q)$) the index set of rows (columns) with fractional numbers, and by $\IRI(Q)$ ($\ICI(Q)$) the index set of rows (columns) with only integral numbers (namely $0$ and $\pm 1$). 

\begin{observation}
    Let $A \in S_n(\ZZ)$ and $Q \in \cQ(A)$. 
    Suppose $P_R$ and $P_C$ are permutation matrices mapping $\FRI(Q)$ and $\FCI(Q)$ to $\set{1,2, \ldots, s}$ respectively, and $P_R^\top Q P_C \in \sCAN(n, h; s)$. 
    Then there exists $Q' \in \cQ(A)$ such that $P_R^\top Q' P_C \in \CAN(n, h; s)$. 
\end{observation}

Note that the number of choices for the pair of permutation matrices $P_R$ and $P_C$ is at most $(n (n-1) \cdots (n-s+1))^2 \leq n^{2s}$. 

We are prepared to prove~\cref{thm:fixed-height}.

\begin{proof}[{Proof of~\cref{thm:fixed-height}}]
    For sufficiently large $n$, we have
    \begin{align*}
        &\Pr(G \text{ has a cospectral mate with height } h) \\
        \leq&\Pr(\exists Q \in O_{n, h}(\QQ) : 2 \leq h, Q^\top A Q \in M_n(\ZZ)) \\
        \leq& \sum_{s = 2}^n \sum_{Q \in \CAN(n, h; s)} n^{2s} \Pr(Q^\top A Q \in M_n(\ZZ)) \\
        \leq& \sum_{s = 2}^n n^{2s} (2s h^2)^{s h^2} \hat{p}^{\frac{s}{2 h^4} \cdot (\frac{s}{2 h^4} + n - s - 1)} \\
        \leq& \sum_{s = 2}^n n^{2s} (2s h^2)^{s h^2} \hat{p}^{\frac{s}{2 h^4} \cdot \frac{n-1}{2 h^4}} \\
        \leq& \sum_{s = 2}^n \exp(2s \ln n + s h^2 \ln (2s h^2) + \frac{s(n-1)}{4h^8} \ln \hat{p}).
    \end{align*}
    Set $\lambda \coloneqq - \ln \hat{p} > 0$ and $a_{n, s} \coloneqq \exp\paren*{2s \ln n + s h^2 \ln (2s h^2) - \frac{\lambda s(n-1)}{4h^8}}$. 
    We will show that $\sum\limits_{s = 2}^n a_{n, s} \to 0$ as $n \to \infty$. 
    Note that for $2 \leq s \leq n$, it holds that
    \begin{align*}
        \ln a_{n, s} &\leq 2s \ln n + s h^2 \ln (2n h^2) - \frac{\lambda s(n-1)}{4 h^8} \\
        &\leq s \paren*{(2+h^2) \ln n + h^2 \ln (2 h^2) - \frac{\lambda (n-1)}{4 h^8}} \eqqcolon s b_n.
    \end{align*}
    Since $\ln n = o(n)$, we have 
    \begin{align*}
        (2+h^2) \ln n + h^2 \ln (2 h^2) \leq \frac{\lambda n}{8 h^8}
    \end{align*}
    for sufficiently large $n$, say $n \geq N_0$. 
    Now suppose $n \geq N_0$, we have
    \begin{align*}
        b_n \leq \frac{\lambda n}{8 h^8} - \frac{\lambda (n-1)}{4 h^8} \leq - \frac{\lambda n}{8 h^8}.
    \end{align*}
    Hence
    \begin{align*}
        \sum_{s = 2}^n a_{n, s} &\leq \sum_{s = 2}^\infty \exp\paren*{- \frac{\lambda n}{8 h^8} s} = \frac{r_n^2}{1 - r_n} \to 0,
    \end{align*}
    as $n \to \infty$, where $r_n = \exp\paren*{- \frac{\lambda n}{8 h^8}} < 1$. 
\end{proof}

It is not hard to generalize the above proof to the case of growing $h$. 

\begin{proof}[{Proof of~\cref{thm:growing-height}}]
    We keep the notation in the proof of~\cref{thm:fixed-height}. 
    Since $h = o(\paren*{n/\ln n}^{1/10})$, we have
    \begin{align}
        b_n \leq - C h^2 \ln n
    \end{align}
    for some $C > 0$. 
    Therefore
    \begin{align}
        \sum_{s=2}^n a_{n, s} \leq \sum_{s=2}^n \exp\paren*{- C s h^2 \ln n} \leq \frac{\exp\paren*{- 2C h^2 \ln n}}{1-\exp\paren*{- C h^2 \ln n}} \to 0
    \end{align}
    since $\exp\paren*{- 2C h^2 \ln n} \to 0$ as $n \to \infty$.
\end{proof}

\section{Discussion}

We exhaust all graphs of order at most $10$, and investigate the maximum level/height such that there exists a controllable graph with generalized cospectral mate of given level/height. 
The result in summarized in~\cref{tab:max_level_height}, where NaN means that there exists no such controllable graph. 

\begin{table}[htbp]
    \centering
    \begin{tabular}{ccccccccccc}
        \hline
        $n$ & 1 & 2 & 3 & 4 & 5 & 6 & 7 & 8 & 9 & 10 \\
        \hline
        $\ell_{\text{max}}$ & NaN & NaN & NaN & NaN & NaN & NaN & NaN & 3 & 37 & 253 \\
        \hline
        $h_{\text{max}}$ & NaN & NaN & NaN & NaN & NaN & NaN & NaN & 3 & 37 & 253 \\
        \hline
    \end{tabular}
    \caption{The maximum level/height such that there exists a controllable graph of order $n$ with generalized cospectral mate of given level/height.}
    \label{tab:max_level_height}
\end{table}

It seems that the maximum level and maximum height coincide for small controllable graphs. 
The author believes that they should diverge as the order of the graph grows. 
The data suggest that we need other methods to handle cospectral mates with large level/height. 
The author believes that tools in number theory will be useful.
It is tentative to bound the number of (canonical) rational orthogonal matrix of given height by the following theorem. 
However, the constant term, which the author is aware of, is relatively large. 

\begin{theorem}[{\cite[Theorem 0.1]{salberger2023CountingRationalPoints}}]
    Let $X \subset \PP^n$ be an integral projective variety of degree $d \geq 2$ defined over $\QQ$. 
    Then for every $\epsilon > 0$, the number of rational points on $X$ with bounded height $h$ is controlled by $O_{X, \epsilon}(h^{\dim X + \epsilon})$.
\end{theorem}

For another approach, consider the set of $n \times n$ square matrices $M_n$ over a field or a ring. 
The ordinary matrix multiplication and Hadamard product(also known as Schur product, entrywise product) are naturally defined on $M_n$. 
A natural graph matrix is obtained from the adjacency matrix by a fixed sequence of operations: (1) linear combinations, (2) ordinary matrix multiplication $\cdot$, (3) Hadamard product $\circ$. 
The identity matrix $I$ and the all-one matrix $J$ are the units of these two products.

Xiang~\cite{xiang2026NaturalGraphSpectra} proved the following theorem.

\begin{theorem}[{\cite[Theorem 1.1]{xiang2026NaturalGraphSpectra}}]
    Let $n$ be a natural number. 
    There exists a natural graph spectrum (which may depend on $n$) such that for the Erdős--Rényi random graph $G(n, 1/2)$, the spectrum determines $G$ up to isomorphism asymptotically almost surely.
\end{theorem}

In other words, there exists a sequence of functions $f_n$ in the double algebra of square matrices such that the spectrum of $f_n(A)$ determines almost all graphs, where $A$ is the adjacency matrix of the graph. 
This leads to the following problems for future research.

\begin{problem}
    How to construct $f_n$ explicitly? How to compute it efficiently?
\end{problem}

\begin{problem}
    How to reduce the `degree' of $f_n$?
\end{problem}

One feasible definition of `degree' is the iteration depths of operations. 

\begin{problem}
    Could we choose $f_n$ such that $f_n(A)$ is always real symmetric?
\end{problem}

It might works by taking $f_n(A) + f_n(A)^\top$. 

\begin{problem}
    Could we find a fixed function $f$ independent of $n$? The Haemers conjecture for adjacency matrix, Laplacian matrix, and generalized adjacency matrix is equivalent to claim $f(x) = x$, $f(x) = (x \cdot x) \circ I - x$, and $f(x) = s_1 x + s_2  (J-I-x)$ respectively. 
\end{problem}

\section*{Declaration of competing interest}

The author declares no known competing financial interests or personal relationships that could have appeared to influence the work reported in this paper.



\section*{Acknowledgements}
The author would like to thank Willem Haemers, Guoyou Qian (千国有), Edwin van Dam, Alexander Van Werde, Wei Wang (王卫), and Wei Wang (王伟) for helpful discussions and comments. 
Da ZHAO (赵达) was supported in part by the National Natural Science Foundation of China (No. 12471324, No. 12501459, No. 12571353), and the Natural Science Foundation of Shanghai, Shanghai Sailing Program (No. 24YF2709000). 

\bibliographystyle{alphaurl}
\bibliography{ref}

\end{document}